\title{Localization and homological stability of configuration spaces}
\author{
        Martin Bendersky and Jeremy Miller
}
\date{\today}

\documentclass[12pt]{article}

\usepackage {amsmath, amscd, amsbsy, amsfonts}
\usepackage{latexsym,amssymb,mathrsfs,bbm}
\usepackage[all]{xy}
\usepackage{graphicx}
\usepackage{latexsym,amssymb,mathrsfs,bbm,enumerate,verbatim}
\usepackage[pagebackref,colorlinks,citecolor=blue,linkcolor=blue,urlcolor=blue,filecolor=blue]{hyperref}
\usepackage{tikz}

\newcounter{prob}
\newcommand{\Q}{\mathbb{Q}}

\newcommand{\Zp}{\mathbb{Z}_{(p)}}
\newcommand{\Z}{\mathbb{Z}}

\newcommand{\s}{S^n_{(0)}}

\newcommand{\m}{\longrightarrow}

\newcommand{\Sp}{S^n_{(p)}}

\newtheorem{theorem}{Theorem}[section]
\newtheorem{lemma}[theorem]{Lemma}
\newtheorem{proposition}[theorem]{Proposition}
\newtheorem{corollary}[theorem]{Corollary}
\newtheorem{remark}[theorem]{Remark}

\newtheorem{definition}{Definition}[section]

\newenvironment{proof}[1][Proof]{\begin{trivlist}
\item[\hskip \labelsep {\bfseries #1}]}{\end{trivlist}}

\newcommand{\qed}{\nobreak \ifvmode \relax \else
      \ifdim\lastskip<1.5em \hskip-\lastskip
      \hskip1.5em plus0em minus0.5em \fi \nobreak
      \vrule height0.75em width0.5em depth0.25em\fi}

\begin{document}
\maketitle

\begin{abstract}
In \cite{Ch}, Church used representation stability to prove that the space of configurations of distinct unordered points in a closed manifold exhibit rational homological stability. A second proof was also given by Randal-Williams in \cite{RW} using transfer maps. We give a third proof of this fact using localization and rational homotopy theory. This gives new insight into the role that the rationals play in homological stability. Our methods also yield new information about stability for torsion in the homology of configuration spaces of points in a closed manifold.

\end{abstract}

\section{Introduction}

Throughout this paper, $M$ will be a smooth connected $n$-manifold with $n \geq 2$. Let $C_k(M)$ denote the configuration space of finite subsets of $M$ of cardinality $k$. That is, $C_k(M) =( M^k -\Delta_{fat} )/\Sigma_k$ where $\Delta_{fat}$ is the fat diagonal and $\Sigma_k$ is the symmetric group. When $M$ is the interior of a manifold with non-empty boundary, McDuff in \cite{Mc1} defined a stabilization map: $$t_k: C_k(M) \m C_{k+1}(M) $$ involving ``bringing a point in from infinity.'' McDuff proved that there is a number $N_k$ depending only on $k$ and $M$ such that $t_k$ induces an isomorphism on groups $H_i( \cdot ;\Z)$ for $i \leq N_k$ and $lim_{k \m \infty} N_k =\infty$  (Theorem 1.2 of \cite{Mc1}). Later Segal showed that one can take $N_k$ to be $k/2$ (Proposition A.1 of \cite{Se}).

The question of homological stability for configuration spaces of particles in closed manifolds was not addressed for over $30$ years until the work of Church (Corollary 3 of \cite{Ch}). A second proof was also given by Randal-Williams (Theorem C of \cite{RW}). There are two main difficulties in studying configuration spaces of particles in closed manifolds. There is no natural map $C_k(M) \m C_{k+1}(M)$ and the integral homology of the spaces $C_k(M)$ do not stabilize (this can be seen by computing $H_1(C_k(S^2))$ from the presentation of $\pi_1(C_k(S^2))$ given on page 255 of \cite{FV}). Nevertheless, Church was able to prove using representation stability that the rational homology of the spaces $C_k(M)$ does stabilize. Randal-Williams later gave a direct proof using transfer maps.

We give an alternative perspective on this phenomenon using localization and construct a zig-zig of maps of spaces between $C_k(M)$ and $C_{k+1}(M)$ which induces an isomorphism in rational homology in a stable range. We hope that this approach might also be useful for proving rational homological stability theorems in situations where there are no natural stabilization maps or even transfer maps.

We also give conditions for $C_k(M)$ and $C_j(M)$ to have isomorphic $p$-torsion. Until this point, the only theorems regarding homological stability for torsion for closed manifolds are due to B{\"o}digheimer, Cohen and Taylor in \cite{BCT} (see also \cite{RW}). Their work yields an explicit calculation of the $\mathbb F_p$-homology of configuration spaces when $p$=2 or the manifold is odd dimensional. From these explicit calculations, one can observe that the homology often stabilizes with $\mathbb F_p$ coefficients. In contrast, our results concern the entire $p$-torsion subgroups of the integral homology of configuration spaces (see Theorem \ref{torsion}). Interestingly, there are situations where the homology with $\mathbb{F}_p$ coefficients stabilizes but the $p$-torsion does not (page 255 of \cite{FV}).

\paragraph{Acknowledgments}

We would like to thank Oscar Randal-Williams and Tom Church for several helpful conversations as well as the referee for many important suggestions.

\section{Scanning}

Although McDuff did not address the question of homological stability for configuration spaces of points in a closed manifold, she did prove the following theorem (Theorem 1.1 of \cite{Mc1}).

\begin{theorem}
Let $M$ be a closed manifold and let $\dot TM \m M$ denote the fiberwise one point compactification of its tangent bundle. There is a map: $$s:C_k(M) \m \Gamma_k(\dot TM)$$ which induces homology isomorphisms through the same range that $t_k: C_k(M -pt) \m C_{k+1}(M -pt)$ is a homology isomorphism. 
\end{theorem}

Here $\Gamma_k(\dot TM)$ is the space of degree $k$ sections of $\dot TM \m M$. See page 102 of \cite{Mc1} for the definition of the degree of a section. The map $s$ is called the scanning map. Combining this theorem with Segal's explicit homological stability range from \cite{Se}, one gets the following corollary.

\begin{corollary}
The map $s:C_k(M) \m \Gamma_k(\dot TM)$ induces an isomorphism on $H_i(\cdot,\Z)$ for $i \leq k/2$. 

\label{scan}
\end{corollary}

\section{Rational homotopy}

From this point onward, we will always assume that $M$ is a closed manifold. By Corollary \ref{scan}, to prove rational homological stability for the spaces $C_k(M)$, it suffices to prove that the spaces $\Gamma_k(\dot TM)$ are rationally homotopic. While this is not always the case, we will prove that $\Gamma_k(\dot TM)$ is rationally homotopic to $\Gamma_j(\dot TM)$ if $n$ is odd or $k$ and $j$ are both not equal to half the Euler characteristic of $M$.

Let $\s$ denote the rational localization of the $n$-sphere and let $\dot TM_{(0)} \m M$ denote the fiberwise rational localization of the bundle $\dot TM \m M$ \cite{Su}. Composition with the fiberwise localization map $l: \dot TM  \m \dot TM_{(0)}$ gives a map of spaces of sections $\Gamma_k(\dot TM) \m \Gamma_k(\dot TM_{(0)})$. By Theorem 5.3 of \cite{Mo}, this map is a localization map and hence induces an isomorphism $H_*(\Gamma_k(\dot TM);\Q) \m  H_*(\Gamma_k(\dot TM_{(0)});\Q)$. While the path components of $\Gamma( \dot TM  )$ are the spaces $\Gamma_k( \dot TM  )$ for $k \in \Z$, the path components of $\Gamma( \dot TM_{(0)}  )$ are the spaces $\Gamma_k( \dot TM_{(0)}  )$ for $k \in \Q$. 

\begin{proposition}
Given a section of $\dot TM_{(0)}$, there is a natural bijection $\Q = \pi_n(\s) \m \pi_0( \Gamma(\dot TM_{(0)}))$.
\label{pizero}
\end{proposition}

\begin{proof}
Pick a $CW$-structure on $M$ with exactly one $n$-cell. Let $M_{n-1}$ be the $(n-1)$-skeleton of $M$ and let $\dot TM^{n-1}_{(0)}$ denote the restriction of $\dot TM_{(0)}$ to $M_{n-1}$. The inclusion $M_{n-1} \m M$ is a cofibration with cofiber equal to $S^n$. Thus, there is a fiber sequence: $$\Omega^n \s \m \Gamma(\dot TM_{(0)}) \m  \Gamma(\dot TM^{n-1}_{(0)}).$$  Here $\Omega^n$ denotes the space of based maps from $S^n$. Note that the map $\Omega^n \s \m \Gamma(\dot TM_{(0)})$ depends on choice of base point section. Since $\s$ is $(n-1)$-connected, $\Gamma(\dot TM^{n-1}_{(0)})$ is connected and so $\Q=\pi_n(\s)=\pi_0(\Omega^n \s) \m \Gamma(\dot TM_{(0)})$ is a surjection.

To see that it is an injection, we will show that $\pi_1( \Gamma(\dot TM_{(0)})) \m \pi_1( \Gamma(\dot TM^{n-1}_{(0)}))$ is a surjection. By Theorem 4.1 of \cite{Mo}, these section spaces are nilpotent. Consider the analogous fiber sequence without localizing. Since $\pi_0(\Omega^n S^n) \m \pi_0(\Gamma(\dot TM))$ is a bijection, $\pi_1( \Gamma(\dot TM)) \m \pi_1( \Gamma(\dot TM^{n-1}))$ is a surjection. Since localization is an exact functor (Proposition 4.6 of \cite{Hi}), $\pi_1( \Gamma(\dot TM_{(0)})) \m \pi_1( \Gamma(\dot TM^{n-1}_{(0)}))$ is also a surjection.

\end{proof}

The topology of these spaces depends heavily on whether $n$ is even or odd. First we will address the case of $n$ odd.

\begin{proposition}
If $n$ is odd, the path components of $\Gamma(\dot TM_{(0)})$ are all homotopic. 
\label{odd}
\end{proposition}

\begin{proof}
Let $\sigma_1$ and $\sigma_2$ be sections of $\Gamma(\dot TM_{(0)})$. Let $P$ be the bundle whose fiber over a point $m \in M$ is the space of degree one self maps of $\dot T_m M_{(0)}$ sending $\sigma_1(m)$ to $\sigma_2(m)$. Note that by Serre's calculations of rational homotopy groups of spheres \cite{Ser}, $\s \simeq K(\Q,n)$ for $n$ odd. Since $Map(X,Y)_{(0)} \simeq Map(X,Y_{(0)}) \simeq Map(X_{(0)},Y_{(0)})$, the fibers of $P$ are homotopic to $\Omega^n_1 \s \simeq \Omega^n_1  K(\Q,n) \simeq *$. Here the subscript $1$ denotes the subspace of degree one maps. Since the fibers are contractible, there are no obstructions to finding a section $f \in \Gamma(P)$  (the relevant obstructions lie in $H^i(M;\pi_{i-1}(\Omega^n_1  K(\Q,n)))$). Instead of thinking of $f$ as a section of $P$, we can view $f$ as a bundle map $f:\dot TM_{(0)} \m \dot TM_{(0)}$ such that $f \circ \sigma_1 =\sigma_2$. To see that $f$ is unique up to fiberwise homotopy, let $f':\dot TM_{(0)} \m \dot TM_{(0)}$ be another such map and let $P'$ be the bundle whose fiber over a point $m \in M$ is the space of maps $H:[0,1] \m P_m$ such that $H(0)=f $ and $H(1)=f'$. The fibers of $P'$ are homotopic to $\Omega^{n+1} \s$ and so there are also no obstructions to finding sections of $P'$. These sections correspond to fiberwise homotopies between $f'$ and $f$. The same argument shows that we can also find a bundle map $g:TM_{(0)} \m TM_{(0)}$ such that $g \circ \sigma_2 =\sigma_1$. Uniqueness shows that $g$ is a fiber homotopy inverse to $f$. Thus, composition with $f$ gives a homotopy equivalence between the path component of $\Gamma(\dot TM_{(0)})$ containing $\sigma_1$ and the component containing $\sigma_2$.
\end{proof}

Before we discuss the case of even dimensional manifolds, we will discuss the case that the tangent bundle is trivial. However, we will make no assumptions regarding the parity of $n$. In this case, we have a natural homeomorphism between $\Gamma_k(\dot TM_{(0)})$ and $Map_k(M,\s)$, the space of degree $k$ maps from $M$ to $\s$. A function $f: M \m \s$ is defined to be degree $k$ if $f_*([M])=k \cdot l_* [S^n]$ with $l:S^n \m \s$ the localization map. We will call $l_* [S^n]$ the fundamental class of $\s$. Note that $M$ has a fundamental class if the tangent bundle is trivial since parallelizable manifolds are orientable.

\begin{lemma}
For $k$ and $j$ non-zero, there is a homotopy equivalence between $Map_k(M,\s)$ and $Map_j(M,\s)$.
\label{map}
\end{lemma}

\begin{proof}
Let $f:\s \m \s$ be a degree $j/k$ map and let $g:\s \m \s$ be a degree $k/j$ map. Since $f \circ g \simeq  g \circ f \simeq id$, composition with $f$ gives a homotopy equivalence between $Map_k(M,\s)$ and $Map_j(M,\s)$ with homotopy inverse given by composition with $g$.
\end{proof}

We will now show that the bundle $\dot TM_{(0)} \m M$ is a trivial bundle if $M$ is orientable.

\begin{definition}
A $\s$-bundle is called orientable if the structure group can be reduced to $Map_1(\s,\s)$.
\end{definition}

\begin{lemma}
Let $E \m M$ be an orientable $S^n_{(0)}$-bundle. Then $E$ is bundle isomorphic to the trivial bundle.
\label{bundleiso}
\end{lemma}

\begin{proof}
Since $E$ is orientable, it is classified by a map to $BMap_1(\s,\s)$. Using a result of Thom from \cite{Th}, M{\o}ller and Raussen (\cite{MR} Example 2.5) observed that: $$ Map_d(\s,\s) \simeq
\begin{cases}
 \s \times S_{(0)}^{n-1} & \text{if }n\text{ is even and } d = 0 \\
S_{(0)}^{2n-1} & \text{if }n\text{ is even and } d \neq 0 \\
\s & \text{if }n\text{ is odd.}
\end{cases}$$ Since $BMap_1(\s,\s) $ is at least $n$-connected and $M$ is $n$ dimensional, the classifying map of $E$ is null-homotopic. 

\end{proof}

If $M$ is orientable, the bundle $\pi: \dot TM_{(0)} \m M$ is also orientable. It is not true that the bundle isomorphism between $\dot TM_{(0)}$ and the trivial bundle necessarily preserves the zero section and hence the homeomorphism $\Gamma(\dot TM_{(0)}) \simeq Map(M,\s)$ might not preserve degree. Note that if $M$ is orientable, the degree of a section of $\dot TM \m M$ is the algebraic intersection number of that section with the zero section.

\begin{proposition}
If $n$ is even and $M$ is orientable, there is a homeomorphism $\Gamma_k(\dot TM_{(0)}) \cong Map_{k-\chi(M)/2}(M,\s)$ induced by a trivialization of $\dot TM_{(0)}$.
\label{even}
\end{proposition}

\begin{proof}

By Lemma \ref{bundleiso}, there exists a trivialization $\tau : \dot TM_{(0)} \m M \times \s$. Let $[B],[F] \in H_n( M \times \s) =\Z \times \Q$ be the generators associated to the fundamental classes of the base and fiber respectively. Let $\bullet : H_n(\dot TM_{(0)}) \times H_n(\dot TM_{(0)}) \m \Q$ denote intersection number. Note that $\tau^{-1}_*([B]) \bullet \tau^{-1}_*([F])=1$, $\tau^{-1}_*([F]) \bullet \tau^{-1}_*([F])=0$ and $\tau^{-1}_*([B]) \bullet \tau^{-1}_*([B])=0$. Let $\sigma_0: M \m \dot TM_{(0)} $ be the zero section. Let $a$ be the number such that $\tau_*(\sigma_{0*}([M])) =[B]+a[F]$. Note that $\chi(M) =$ $$ \sigma_{0*}([M]) \bullet \sigma_{0*}([M]) = (\tau^{-1}_*([B])+a\tau^{-1}_*([F)])) \bullet (\tau^{-1}_*([B])+a\tau^{-1}_*([F]))=2a$$ and so $a = \chi(M)/2$.

Now let $\sigma$ be an arbitrary degree $k$ section and let $b$ be the number such that $\tau_*(\sigma_*([M])) =[B]+b[F]$.  Note that $k =$ $$\sigma_{0*}([M]) \bullet \sigma_{*}([M]) = (\tau^{-1}_*([B])+ (\chi(M)/2)\tau^{-1}_*([F])) \bullet (\tau^{-1}_*([B])+b\tau^{-1}_*([F]))$$ $$=\chi(M)/2+b.$$ Thus $b = k- \chi(M)/2$. Since $\tau \circ \sigma \in Map_b(M,\s)$, the claim follows.

\end{proof}

We now address the case of non-orientable manifolds.

\begin{proposition}
If $n$ is even and $M$ is not orientable, there is a homotopy equivalence $\Gamma_k(\dot TM_{(0)}) \simeq \Gamma_j(\dot TM_{(0)})$ for all $k$ and $j$ not equal to $\chi(M)/2$.
\label{evenNonOr}
\end{proposition}

\begin{proof}
For every non-zero $d \in \Q$, one can construct a bundle automorphism $f_d : \dot TM_{(0)} \m \dot TM_{(0)}$ which induces a map of degree $d$ on each fiber. This follows from obstruction theory since the relevant obstructions lie in $H^i(M;\pi_{i-1}(Map_d(\s,\s)))$ for $i \leq n$ and these groups vanish since $Map_d(\s,\s)$ is $(2n-2)$-connected. Since $\pi_{i}(Map_d(\s,\s))$ vanishes for $i \leq n$, we have a unique, up to fiberwise homotopy, bundle map of a given degree. Thus, $f_d \circ f_{1/d}$ is fiberwise homotopic to the identity.

The bundle maps $f_d$ induce homotopy equivalences $\Gamma_k( \dot T M_{(0)}) \m \Gamma_q( \dot T M_{(0)})$ for some number $q\in \Q$. Our goal is to show that $q=dk+(1-d)\chi(M)/2$. If we could establish this, then the bundle maps would induce homotopy equivalences between every component of $\Gamma(\dot TM_{(0)})$ except for the degree $\chi(M)/2$ component. Let $\tilde M$ denote the orientation double cover of $M$. Since the tangent bundle of $\tilde M$ is the pull back of the tangent bundle of $M$, we can lift degree $k$ sections of $\dot T M_{(0)}$ to degree $2k$ sections of $\dot T \tilde M_{(0)}$ and bundle maps $f_d$ to bundle maps $\tilde f_d: \dot T \tilde M_{(0)} \m \dot T \tilde M_{(0)}$ which also induce degree $d$ maps on each fiber. It follows from Proposition \ref{even} that $\tilde f_d$ induces a map $\Gamma_{2k}(\dot T \tilde M_{(0)}) \m \Gamma_{2kd+(1-d) \chi(\tilde M)/2 }(\dot T \tilde M_{(0)})$. Since $\chi(\tilde M)/2= \chi(M)$, composition with $f_d$ gives a map between $\Gamma_k( \dot T M_{(0)})$ and $\Gamma_{dk+(1-d)\chi(M)/2}( \dot T M_{(0)})$. Since these maps are homotopy equivalences, $\Gamma_k(\dot TM_{(0)}) \simeq \Gamma_j(\dot TM_{(0)})$ for all $k$ and $j$ not equal to $\chi(M)/2$.

\end{proof}

Combining Proposition \ref{odd}, Lemma \ref{map}, Proposition \ref{even} and Proposition \ref{evenNonOr}, we get the following corollary.

\begin{theorem}
The rational homology of $\Gamma_k(\dot TM)$ is isomorphic to the rational homology of $\Gamma_j(\dot TM)$ unless $n$ is even and $k$ or $j$ is $\chi(M)/2$.

\label{sec}
\end{theorem}

Combining Corollary \ref{scan} and Theorem \ref{sec}, we deduce homological stability for configuration spaces of points in a closed manifold. 

\begin{corollary}
The homology groups $H_i(C_k(M);\Q)$ are equal to those of $H_i(C_j(M);\Q)$ if $i \leq min(k/2,j/2)$ and $k,j \neq \chi(M)/2$. Moreover, an isomorphism is given by traversing the following diagram:
$$
\begin{array}{ccccccccl}
C_{k}(M) & \overset{s}{\m} &  \Gamma_k(\dot TM) & \overset{l}{\m} & \Gamma_k(\dot TM_{(0)})  \\

  &   & &   &  \downarrow \simeq            \\

C_{j}(M) & \overset{s}{\m} &  \Gamma_j(\dot TM) & \overset{l}{\m} & \Gamma_{j}(\dot TM_{(0)}) .\\ 

\end{array}$$ 

\end{corollary} 

\begin{remark}
Since the theorems of Church in \cite{Ch} and Randal-Williams \cite{RW} apply equally well to the component $C_{\chi(M)/2}$, one can rephrase the results of this section as follows. For any $k \in \Z$ and an orientable $n$-manifold $M$ of even Euler characteristic, the groups $H_i(Map_{0}(M,S^n);\Q)$ are isomorphic to $H_i(Map_{k}(M,S^n);\Q)$ for $i < \chi(M)/2$. If $n>2$, then the range can be extended to all $i \leq \chi(M)/2$ \cite{RW}. Also note that for $n$ odd, all components of $Map(M,S^n)$ are rationally homotopic.

\end{remark}

\section{Torsion}

In this section, we describe how to modify the arguments of the previous section to compare the torsion in the homology of components of spaces of sections or configuration spaces. First we discuss the connectivity of the spaces of self maps of $p$-local spheres. Then we consider the case when $n$ is odd. Following that, we describe when the $p$-torsion of $Map_k(M,S^n)$ is isomorphic to the $p$-torsion of $Map_j(M,S^n)$ and give a method for comparing the $p$-torsion in the homology of spaces of maps and spaces of sections. Finally, we draw new conclusions about stability for torsion in the homology of configuration spaces of particles in closed manifolds. Let $\Z_{(p)}$ denote the $p$-local integers, $S^n_{(p)}$ the $p$-local $n$-sphere and $\dot TM_{(p)}$ the fiberwise $p$-localization. As in the rational case, Theorem 5.3 of \cite{Mo} implies that fiberwise $p$-localization induces a localization map on spaces of sections. Using a similar argument to those used to prove Proposition \ref{pizero}, we see that $\pi_0(\Gamma(\dot TM_{(p)}))$ can be identified with $\Z_{(p)}$.

Many facts about rational localizations of spaces are also true for $p$-localizations for $p$ sufficiently large. For example, in the previous section, we often used the fact that $\Omega_1^n \s$ and $Map_1(\s,\s)$ are highly connected. This generalizes as follows.

\begin{proposition}
Let $p \geq n/2+3/2$ be a prime. If $n$ is odd then $\Omega_1^n \Sp$ is $(n-1)$-connected. If $n$ is even, then $Map_1(\Sp,\Sp)$ is $(n-1)$-connected.
\label{torCon}
\end{proposition}

\begin{proof}
First consider the case that $n$ is odd. Recall $\Omega_1^n \s$ is weakly contractible since $\s \simeq K(\Q,n)$. Thus, to prove $\Omega_1^n \Sp$ is $(n-1)$-connected, it suffices to prove that $\pi_i(\Omega^n S^n)$ has no $p$ torsion for $i \leq n-1$ and $p \geq n/2+3/2$. This follows by Serre's calculation in \cite{Ser} of the first time $p$-torsion appears in the homotopy groups of spheres. 

Now assume that $n$ is even. Since $Map_1(\s,\s)$ is $(n-1)$-connected, it suffices to prove $\pi_i(Map_1(S^n,S^n))$ has no $p$-torsion for $i \leq n-1$ and $p \geq n/2+3/2$. Consider the following fibration: $$ \Omega^n_1 \Sp \m Map_1(\Sp,\Sp) \m \Sp.$$ The homotopy groups $\pi_i(\Sp)$ vanish $i \leq n-1$. Again by Serre's work in \cite{Ser}, $\pi_i(\Omega^n_1 S^n)$ are $p$-torsion free for $i \leq n-1$. Since $\pi_i(\Omega^n_1 \s)$ is zero for $i \leq n-2$ and $\pi_{n-1}(\Omega^n_1 \s) =\Q$ , we have that $\pi_i(\Omega^n_1 \Sp)=0$ for $i \leq n-2$ and $\pi_{n-1}(\Omega^n_1 \Sp) =\Zp$. From this, we can conclude that $\pi_i(Map_1(\Sp,\Sp))=0$ for $i \leq n-2$. To see that $\pi_{n-1}(Map_1(\Sp,\Sp))=0$, consider the following portion of the long exact sequence of homotopy groups associated to the above fibration: $$\pi_{n}(\Sp) \m \pi_{n-1}(\Omega_1^n \Sp) \m \pi_{n-1}(Map_1(\Sp,\Sp)) \m \pi_{n-1}(\Sp).$$ Note that $\pi_n(\Sp)=\pi_{n-1}(\Omega^n_1 \Sp) =\Zp$ and $\pi_{n-1}(\Sp)=0$. Thus, the vanishing of $\pi_{n-1}(Map(\Sp,\Sp))$ is equivalent to the surjectivity of the connecting homomorphism $\delta: \pi_n(\Sp) \m \pi_{n-1}(\Omega^n_1 \Sp)$. 

Let $G \in \pi_{n-1}(\Omega^n S^n)$ generate a $\Z$ summand. Let $\iota$ denote the generator of $\pi_n(S^n)$.  Let $k$ be the number such that $kG = \delta(\iota)$  modulo torsion. By Theorem 3.2 of \cite{W1}, the map $\delta$ is equal to the map which sends an element $\kappa \in \pi_n(S^n)$ to the Whitehead product of $\kappa$ with $\iota$, $[\kappa,\iota]\in \pi_{n-1}(\Omega^n_1 S^n)$. Thus $\delta(\iota)=[\iota,\iota]$. The Hopf invariant homomorphism $h:\pi_{n-1}(\Omega^n_1 S^n) \m \Z$ sends $[\iota,\iota]$ to $\pm 2$ so $k=\pm 1$ or $\pm 2$. Since $p>2$, $2$ is invertible in $\Zp$ and so the connecting homomorphism $\delta: \pi_n(\Sp) \m \pi_{n-1}(\Omega^n_1 \Sp)$ is surjective. This completes the proof.

\end{proof}

The following lemma is a substitute for the fact that, after rationalizing, there is a unique up to homotopy bundle map of a given fiber degree.

\begin{lemma}
Let $p \geq n/2+3/2$ be a prime, $M$ an $n$-manifold, $E \m X$ be an $\Sp$-bundle. Let $f:E \m E$ be a map of fiber bundles that induces a degree $1$ map on each fiber. If $n$ is odd, make the additional assumption that $f$ fixes some section $\sigma: M \m E$. Then $f$ is a fiberwise homotopy equivalence. Moreover, there is some number $N$ such that $f^N$ is fiberwise homotopic to the identity.
\label{unique}
\end{lemma}

\begin{proof}
For $n$ even, let $P$ be the bundle whose fiber over a point $m \in M$ is $Map_1(E_m,E_m)$. For $n$ odd, let $P$ be the bundle with fiber the subspace of maps fixing the section $\sigma$. These assemble to form a bundle of monoids with respect to composition and this product gives $\pi_0(\Gamma(P))$ the structure of a monoid. We will prove that this monoid is in fact a torsion group. 

Give $M$ a $CW$-structure with one $n$-cell. Let $M_{n-1}$ denote the $n-1$ skeleton of $M$ and $P^{n-1}$ the restriction of $P$ to $M_{n-1}$. By Proposition \ref{torCon}, the fibers of $P$ are $(n-1)$-connected and so $\pi_0(\Gamma(P^{n-1}))=1$. The inclusion $M_{n-1} \m M$ is a cofibration with cofiber homotopic to $S^n$. Thus the restriction map $r:\Gamma(P) \m \Gamma(P^{n-1})$ is a fibration. All of the fibers of $r$ can be identified with $\Omega^n P_m$ (the fiber of $P$ at a fixed point $m \in M$) since the fibers of $P$ are $(n-1)$-connected. Thus the following sequence is an exact sequence of monoids: $$ \pi_0(\Omega^n P_m) \m \pi_0(\Gamma(P)) \m 1.$$  The monoid structure on the set $\pi_0(\Omega^n P_m)$ making $\pi_0(\Omega^n P_m) \m \pi_0(\Gamma(P))$ a map of monoids is the monoid structure induced by the monoid structure on $P_m$. Loop sum also gives the set $\pi_0(\Omega^n P_m)$ is a monoid structure. By the Eckmann-Hilton argument, these two monoid structures agree. Note that the rational localization of $P_m$ is $\Omega_1^n \s$ for $n$ odd and $Map_1(\s,\s)$ for $n$ even. Since $P_m$ is rationally $n$-connected, $\pi_{n}(P_m)=\pi_0(\Omega^n P_m)$ is a torsion group. It follow by exactness of the sequence that  $\pi_0(\Gamma(P))$ is also a torsion group. Thus $f$ is homotopy idempotent and hence a homotopy equivalence.

\end{proof}

We can now state a sufficient condition for the path components of $\Gamma(\dot TM_{(p)})$ to be homotopic.

\begin{proposition}
If $n$ is odd and $p \geq n/2+3/2$, the path components of $\Gamma(\dot TM_{(p)})$ are all homotopic. 
\label{oddSec}
\end{proposition}

\begin{proof}
The proof follows the proof of Proposition \ref{odd}. Let $\sigma_1$ and $\sigma_2$ be two sections of $\dot TM_{(p)}$. The obstructions to finding a degree one bundle map $f:\dot TM_{(p)} \m \dot TM_{(p)}$ taking $\sigma_1$ to $\sigma_2$ lie in $H^i(M;\pi_{i-1}(\Omega_1^n S^n_{(p)}))$ for $i \leq n$. Thus the obstructions vanish and one can construct a degree one bundle map taking $\sigma_1$ to $\sigma_2$. By Lemma \ref{unique}, these bundle maps are fiberwise homotopy equivalences. Thus the path components of $\Gamma(\dot TM_{(p)})$ containing $\sigma_1$ and $\sigma_2$ are homotopic.
\end{proof}

The proof of Lemma \ref{map} works with minimal modification to show the following proposition. 

\begin{proposition} If $k/j$ is a unit in $\Z_{(p)}$, then $Map_k(M,S_{(p)}^n)$ and $Map_j(M,S_{(p)}^n)$ are homotopic.

\label{mapTor}
\end{proposition}



The above proposition immediately applies to configuration spaces of particles in parallelizable manifolds. To use it to study non-parallelizable manifolds, we will need to adapt Lemma \ref{bundleiso}, Proposition \ref{even} and Proposition \ref{evenNonOr} to the case of localizing away from a prime $p$. First we generalize Lemma \ref{bundleiso}.

\begin{proposition} If $p \geq n/2 +3/2$ and $M$ is orientable, then $\dot TM_{(p)}\m M$ is the trivial $\Sp$-bundle. If $M$ is an orientable surface, then $\dot TM \m M$ is also trivial.
\label{tortrivial}
\end{proposition}

\begin{proof}
First assume that $p \geq n/3 +3/2$ and $n$ is even. By Proposition \ref{torCon}, $Map_1(S^n_{(p)},S^n_{(p)})$ is $(n-1)$-connected and so $BMap_1(S^n_{(p)},S^n_{(p)})$ is $n$-connected. Thus, the map classifying $\dot TM_{(p)}$ is null-homotopic. If $n$ is odd, a straightforward generalization of Proposition \ref{torCon} gives the result. However, since we will only use Proposition \ref{tortrivial} for $n$ even, we will not provide details for $n$ odd.

Now let $n=2$. Since $BMap_1(S^2,S^2)$ is simply connected, any map $M \m BMap_1(S^2,S^2)$ is homotopically trivial when restricted to the $1$-skeleton. This shows that the set of based maps up to homotopy, $\big [M,BMap_1(S^2,S^2) \big ]$, is isomorphic to a quotient of the group $\big [ S^2, BMap_1(S^2,S^2) \big ] = \pi_2(BMap_1(S^2,S^2))$. By the work of Hu in \cite{Hu}, $\pi_2(BMap_1(S^2,S^2))=\Z/2\Z$. Fiberwise one point compactifying vector bundles induces a map $\big [M,BSO(2) \big ] \m \big [M,BMap_1(S^2,S^2) \big ]$. The Euler class gives an isomorphism between $\big [M,BSO(2) \big ]$ and $\Z$. Since orientable surfaces have even Euler class, the tangent bundle of $M$ gives the trivial element of $\big [M,BMap_1(S^2,S^2) \big ]$ and hence the sphere bundle associated to the tangent bundle is trivial.

\end{proof}

Using Proposition \ref{torCon}, Lemma \ref{unique}, Proposition \ref{mapTor} and Proposition \ref{tortrivial}, we can adapt the proofs of Proposition \ref{even} and Proposition \ref{evenNonOr} to prove the following proposition. 

\begin{proposition}
Assume that $(2k-\chi(M))/(2j-\chi(M))$ is a unit in $\Z_{(p)}$ and $n$ is even. Also assume that either $p \geq n/2+3/2$ or $M$ is an orientable surface. Then there is a homotopy equivalence $\Gamma_k(\dot TM_{(p)}) \simeq \Gamma_j(\dot TM_{(p)})$.
\label{evenSec}
\end{proposition}

\begin{proof}
The case that $M$ is orientable immediately follows from the arguments of Proposition \ref{even} using Proposition \ref{mapTor} and Proposition \ref{tortrivial}.

Now assume $p \geq n/2+3/2$. Let $d$ be a unit in $\Zp$. By Proposition \ref{mapTor}, $Map_1(\Sp,\Sp) \simeq Map_d(\Sp,\Sp)$. By Proposition \ref{torCon}, $Map_1(\Sp,\Sp)$ is $(n-1)$-connected and hence $Map_d(\Sp,\Sp)$ is as well. By obstruction theory, there is a bundle map $f_d$ of fiber degree $d$ since the obstructions lie in $H^i(M;\pi_{i-1}(Map_d(\Sp,\Sp)))$ for $i \leq n$. Note that $f_d$ is a fiberwise homotopy equivalence by Lemma \ref{unique} since $f_d \circ f_{1/d}$ is degree one. 

To make the rest of the arguments of Proposition \ref{evenNonOr} work in this case, we need to know the following fact: Let $d$ be a unit in $\Z_{(p)}$ and let $f_d$ and $f_d'$ be two fiber degree $d$ maps. If $M$ is orientable, then for any section $\sigma$, $f_d \circ \sigma$ and $f'_d \circ \sigma$ have the same degree. To prove this, let $g$ be the composition of $f_d$ with a fiber homotopy inverse of $f'_d$. The claim will follow by showing that $g$ induces the identity on $H_n(\dot TM_{(p)})= H_n(M \times S^n_{(p)}) =\Z \oplus \Z_{(p)}$ as the degree of $g \circ \sigma$ only depends on the degree of $\sigma$ and the map  $g_*:H_n(\dot TM_{(p)}) \m H_n(\dot TM_{(p)})$. Use the basis for $H_n(M \times S^n_{(p)})$ given by the fundamental classes of $M$ and $S^n_{(p)}$. Let $\begin{bmatrix}
 \alpha & \beta \\
 \gamma & \delta
\end{bmatrix}$ be the matrix associated to $g_* :H_n(M \times S^n_{(p)}) \m H_n(M \times S^n_{(p)}) $ in this basis. Since $g$ is a bundle map, $\alpha$=1 and $\beta=0$. Since $g$ has fiber degree equal to $1$, $\delta=1$. By Lemma \ref{unique}, there is some number $N$ such that $g^N$ is fiberwise homotopic to the identity and so $\gamma=0$. This fact is the substitute for the fact that, after rationalizing, there is a unique bundle automorphism inducing a degree $d$ map on each fiber. Everything else is a straightforward adaptation of the arguments of  Proposition \ref{evenNonOr}.
\end{proof}

Using Corollary \ref{scan}, Proposition \ref{oddSec}, Proposition \ref{mapTor} and Proposition \ref{evenSec}, we get the following theorem.

\begin{theorem}
Let $i \leq min(k/2,j/2)$. Then the $p$-torsion of $H_i(C_k(M))$ and $H_i(C_j(M))$ are isomorphic if at least one of the following four conditions are met:

1) $M$ is parallelizable and $k/j$ is a unit in $\Z_{(p)}$ 

2) $n$ is odd and $p \geq n/2+3/2$

3) $n$ is even, $p \geq n/2+3/2$ and $(2k-\chi(M))/(2j -\chi(M))$ is a unit in $\Z_{(p)}$

4) $n=2$, $M$ is orientable and $(2k-\chi(M))/(2j -\chi(M))$ is a unit in $\Z_{(p)}$.

\label{torsion}
\end{theorem}

For example, part 1 of the above theorem implies that the $2$-torsion of the homology of $C_{2k+1}(M)$ stabilizes for $M$ parallelizable. This contrasts with the case of the sphere where part 4 indicates that the $2$-torsion of $H_*(C_{2k}(S^2))$ stabilizes. Fadell and Van Buskirk's calculation that $H_1(Br_k(S^2)) =H_1(C_k(S^2)) = \Z /(2k-2) \Z$ (page 255 of \cite{FV}) shows that the two torsion in $H_*(C_{2k+1}(S^2))$ does not stabilize. So in some sense, Theorem \ref{torsion} is optimal in this case.

\bibliography{thesis2}{}

\def\cprime{$'$}
\begin{thebibliography}{BCT89}

\bibitem[BCT89]{BCT}
C.-F. B{\"o}digheimer, F.~Cohen, and L.~Taylor.
\newblock On the homology of configuration spaces.
\newblock {\em Topology}, 28(1):111--123, 1989.

\bibitem[Chu12]{Ch}
Thomas Church.
\newblock Homological stability for configuration spaces of manifolds.
\newblock {\em Invent. Math.}, 188(2):465--504, 2012.

\bibitem[FVB62]{FV}
Edward Fadell and James Van~Buskirk.
\newblock The braid groups of {$E^{2}$} and {$S^{2}$}.
\newblock {\em Duke Math. J.}, 29:243--257, 1962.

\bibitem[Hil73]{Hi}
Peter Hilton.
\newblock Localization and cohomology of nilpotent groups.
\newblock {\em Math. Z.}, 132:263--286, 1973.

\bibitem[Hu46]{Hu}
Sze-tsen Hu.
\newblock Concerning the homotopy groups of the components of the mapping space
  {$Y^{S^{p}}$}.
\newblock {\em Nederl. Akad. Wetensch., Proc.}, 49:1025--1031 = Indagationes
  Math. 8, 623--629 (1946), 1946.

\bibitem[McD75]{Mc1}
D.~McDuff.
\newblock Configuration spaces of positive and negative particles.
\newblock {\em Topology}, 14:91--107, 1975.

\bibitem[M{\o}l87]{Mo}
Jesper~Michael M{\o}ller.
\newblock Nilpotent spaces of sections.
\newblock {\em Trans. Amer. Math. Soc.}, 303(2):733--741, 1987.

\bibitem[MR85]{MR}
Jesper~Michael M{\o}ller and Martin Raussen.
\newblock Rational homotopy of spaces of maps into spheres and complex
  projective spaces.
\newblock {\em Trans. Amer. Math. Soc.}, 292(2):721--732, 1985.

\bibitem[RW13]{RW}
Oscar Randal-Williams.
\newblock Homological stability for unordered configuration spaces.
\newblock {\em Quarterly Journal of Mathematics}, (64 (1)):303--326, 2013.

\bibitem[Seg79]{Se}
Graeme Segal.
\newblock The topology of spaces of rational functions.
\newblock {\em Acta Math.}, 143(1-2):39--72, 1979.

\bibitem[Ser51]{Ser}
Jean-Pierre Serre.
\newblock Homologie singuli\`ere des espaces fibr\'es. {A}pplications.
\newblock {\em Ann. of Math. (2)}, 54:425--505, 1951.

\bibitem[Sul74]{Su}
Dennis Sullivan.
\newblock Genetics of homotopy theory and the {A}dams conjecture.
\newblock {\em Ann. of Math. (2)}, 100:1--79, 1974.

\bibitem[Tho56]{Th}
Ren\'e Thom.
\newblock L'homologie des espaces fonctionelles.
\newblock pages 29--39. Thone Liège; Masson, Paris, 1956.

\bibitem[Whi46]{W1}
George~W. Whitehead.
\newblock On products in homotopy groups.
\newblock {\em Ann. of Math (2)}, 47:460--475, 1946.

\end{thebibliography}
\bibliographystyle{alpha}

\end{document}